\documentclass[10pt]{amsart}
\usepackage[english]{babel}
\usepackage{amssymb,amsmath,amsthm}
\usepackage{verbatim}
\usepackage{graphicx}

\newtheorem{theorem}{Theorem}
\newtheorem{lemma}{Lemma}
\newtheorem*{definition}{Definition}
\newtheorem{proposition}{Proposition}
\newtheorem*{corollary}{Corollary}

\def\bal{\begin{aligned}}
\def\eal{\end{aligned}}
\def\be{\begin{equation}\label}
\def\ee{\end{equation}}
\def\bcs{\begin{cases}}
\def\ecs{\end{cases}}
\def\={\;=\;}
\def\+{\,+\,}
\def\-{\,-\,}

\def\Z{{\mathbb Z}}
\def\N{{\mathbb N}}
\def\Q{{\mathbb Q}}
\def\R{{\mathbb R}}
\def\F{{\mathbb F}}
\def\P{{\mathbb P}}
\def\L{\Lambda}
\def\Newt{{\rm Newt}}
\def\con{\;\equiv\;}

\def\B{{\mathcal B}}
\def\D{{\mathcal D}}
\def\O{{\mathcal O}}

\title[Dwork's congruences]{Dwork's congruences for the constant terms of powers of a Laurent polynomial}
\author{Anton Mellit, Masha Vlasenko}

\begin{document}
\date{\today}
\maketitle

\begin{abstract}
We prove that the constant terms of powers of a Laurent polynomial satisfy certain congruences modulo prime powers. As a corollary, the generating series of these numbers considered as a function of a p-adic variable satisfies a non-trivial analytic continuation property, similar to what B.~Dwork showed for a class of hypergeometric series.
\end{abstract}

\section{Congruences}

We shall prove the following

\begin{theorem}\label{Th} Let $\L(x) \in \Z_p[x_1^{\pm1},\dots,x_d^{\pm1}]$ be a Laurent polynomial, and consider the sequence of the constant terms of powers of $\L$  
\[
b_n \= \Bigl[ \L(x)^n \Bigr]_0 \,, \quad n=0,1,2,\dots
\] 
Define 
\[
f(X) \= \sum_{n=0}^{\infty} b_n X^n
\]
and 
\[
f_s(X) \= \sum_{n=0}^{p^s-1} b_n X^n\,, \quad s =0,1,2,\dots
\]

If the Newton polyhedron of $\L$ contains the origin as its only interior integral point, then for every $s \ge 1$ one has the congruence 
\be{c1}
\frac{f(X)}{f(X^p)} \;\equiv\; \frac{f_s(X)}{f_{s-1}(X^p)} \mod\; p^s \Z_p[[X]]\,,
\ee
or, equivalently, for every $s \ge 1$
\be{c2}
f_{s+1}(X)f_{s-1}(X^p) \;\equiv\; f_s(X)f_{s}(X^p) \mod\; p^s\Z_p[X]\,.
\ee
\end{theorem}

\vskip1cm

One can easily see that congruence~\eqref{c1} with $s=1$ is equivalent to the statement that for every $n$
\[
b_n \con b_{n \,{\rm mod}\, p} \; b_{ \lfloor\frac{n}{p}\rfloor} \mod\; p \,,
\]
or if we expand $n$ to the base $p$ as $n = n_0 + n_1 p + \dots + n_r p^r = \overline{n_0 \dots n_r}$ with digits $0 \le n_i \le p-1$ then 
\[
b_{n} \con b_{n_0} \, \dots \, b_{n_r} \mod\; p\,.
\]
In \cite{SvS09} Duco van Straten and Kira Samol gave a generalization modulo higher powers of $p$: under the same assumptions as in Theorem~\ref{Th} one has
\be{dig2}
b_{n + mp^s} \; b_{ \lfloor\frac{n}{p}\rfloor } \equiv b_{n} \; b_{ \lfloor\frac{n}{p}\rfloor + m p^{s-1}} \mod\; p^{s} 
\ee
for all $n,m \ge 0$, $s \ge 1$. We do not know whether it is possible to deduce congruences~\eqref{c1}-\eqref{c2} from~\eqref{dig2}. Our method of proof is independent and actually allows one to get~\eqref{dig2} as a byproduct. In fact the main idea used here appeared as an attempt to give an independent proof of ~\eqref{dig2}, and later we realized that it can also be applied to~\eqref{c1}-\eqref{c2}. 

\vskip1cm

Throughout the paper we assume $p$ to be a fixed prime number. For a natural number $n \in \N$ we denote by $\ell(n)=\lfloor \frac{\log n}{\log p} \rfloor + 1$ the length of the expansion of $n$ to the base $p$, and we assume $\ell(0)=1$. For any tuple of non-negative integers $n^{(1)},\ldots,n^{(r)}$ with $n^{(r)}\neq 0$ we introduce the notation 
\[
n^{(1)}*\cdots*n^{(r)}:=n^{(1)} + n^{(2)} p^{\ell(n^{(1)})} + \cdots+n^{(r)} p^{\ell(n^{(1)})+\cdots+\ell(n^{(r-1)})},
\]
that is the expansion of $n^{(1)}*\cdots*n^{(r)}$ to the base $p$ is the concatenation of the respective expansions of $n^{(1)},\ldots,n^{(r)}$.

The proof of Theorem~\ref{Th} is based on the following 

\begin{lemma}\label{lem} Under the assumptions of Theorem~\ref{Th}, there exists a $\Z_p$-valued sequence $\{ c_n; n \ge 0 \}$ such that for all $n \ge 1$
\be{A1}
b_n \=  \underset{n^{(1)} * \dots * n^{(r)} = n} \sum c_{n^{(1)}} \cdot \ldots \cdot c_{n^{(r)}}\,,
\ee
where the sum runs over all $1 \le r \le \ell(n)$ and all possible partitions of the expansion of $n$ to the base $p$ into $r$ expansions of non-negative integers, and
\be{A2}
c_n \equiv 0 \mod\; p^{\ell(n)-1}\,.  
\ee
\end{lemma}

\vskip1cm

The paper is organized as follows. We construct the sequence $\{ c_n ; n \ge 0 \}$ in Sections~\ref{sec:ghost}-\ref{sec:onept}. Section~\ref{sec:proof} can be read independently of the previous three, we deduce Theorem~\ref{Th} from Lemma~\ref{lem} in there. 

\vskip1cm

In the remainder of this section we would like to suggest an application of the congruences stated in Theorem~\ref{Th}. Basically, we extract the following lemma from~\cite{Dw69} (see Theorem 3). But since our setup is simpler and assumptions look slightly different, we give a proof nevertheless. Let us fix the following notation:
\[\bal
& |\cdot|_p \quad \text{ denotes the $p$-adic norm, chosen so that } |p|_p \= p^{-1} \\
& \Omega \quad\= \text{ completion of the algebraic closure of } \Q_p \\
& \O \quad\= \text{ring of integers of } \Omega \= \{ z \in \Omega \;:\; |z|_p \le 1 \} \\
& \B \quad\= \text{ideal of non units in } \O \= \{ z \in \Omega \;:\; |z|_p < 1 \} \\
\eal\]

\begin{lemma}[Dwork]\label{dw_lemma} Let a $\Z_p$-valued sequence $\{ b_n; n \ge 0 \}$ be such that $b_0$ is a unit and congruence~\eqref{c1} holds true for every $s \ge 1$. Consider the region 
\[
\D \= \bigl\{ z \in \O \;:\; |f_1(z)|_p \= 1 \bigr\}\,.
\]  
Then 
\begin{itemize}
\item[(i)] $\D$ contains $\B$, and if $z \in \D$ then $z^p \in \D$;
\item[(ii)] for every $s \ge 0$ one has $|f_s(z)|_p=1$ when $z \in \D$;
\item[(iii)] the sequence of rational functions $f_s(z)/f_{s-1}(z^p)$ converges uniformly in $\D$, and if we denote the limiting analytic function by $\omega(z) = \underset{s \to \infty}\lim f_s(z)/f_{s-1}(z^p)$ then for all $s \ge 1$
\[
\underset{z \in \D}\sup \; \Bigl| \omega(z) \- \frac{f_s(z)}{f_{s-1}(z^p)} \Bigr|_p \;\le\; \frac1{p^s}\,;
\]
\item[(iv)] $f(z)/f(z^p)$, which is a power series with integral coefficients and hence an analytic function on $\mathcal B$, is the restriction of $\omega(z)$ to $\mathcal B$.
\end{itemize}
\end{lemma}
\begin{proof} As $b_0$ is a unit, for $z \in \B$ we have $|f_1(z)|_p = |b_0|_p = 1$ by the isosceles triangle principle for non-Archimedean norms. Since $f_1(X) \in \Z_p[X]$ then $f_1(X)^p-f_1(X^p) \in p \Z_p[X]$ and therefore $|f_1(z)^p-f_1(z^p)|_p \le \frac1p$ for any $z \in \O$. Hence for $z \in \D$ we have $|f_1(z^p)|_p = |f_1(z)^p|= 1$ again by the isosceles triangle principle, so $z^p \in \D$. (ii) follows from the same argument by induction on $s$, since for every $s \ge 1$ we have $f_{s}(X) - f_1(X)f_{s-1}(X^p) \in p \Z_p[X]$. To prove (iii) we notice that $f_{k}(X)f_{s-1}(X^p) - f_{k-1}(X^p)f_s(X) \in p^s \Z_p[X]$ for any $k \ge s$, which together with (ii) gives
\[
\Bigl| \frac{f_k(z)}{f_{k-1}(z^p)} \- \frac{f_s(z)}{f_{s-1}(z^p)} \Bigr|_p \le \frac1{p^s} \qquad \forall z \in \D \,,
\] 
and we see that this sequence of functions is a Cauchy sequence. To prove (iv) observe that  
\[
\Bigl| \frac{f(z)}{f(z^p)} \- \frac{f_s(z)}{f_{s-1}(z^p)} \Bigr|_p \le \frac1{p^s}
\]
for any $z \in \B$ as $f(X)/f(X^p)-f_s(X)/f_{s-1}(X^p) \in p^s\Z_p[[X]]$.
\end{proof}

\vskip0.5cm

Let us take for example the Laurent polynomial
\[
\L(x_1,x_2) \= \frac{(1+x_1)(1+x_2)(1+x_1+x_2)}{x_1 x_2} \,.
\]
One can show that the sequence of the constant terms of its powers is the Ap\'{e}ry sequence
\[
b_n \= \sum_{k=0}^n \binom{n}{k}^2\binom{n+k}{k}\,.
\]
Conditions of Theorem~\ref{Th} are satisfied for all primes $p$ since coefficients are integral and the Newton polygon of $\Lambda$ has one interior integral point. Normalizations of the smooth fibers of the family  
\[
E_t \;:\; \L(x_1,x_2) \= \frac1{t}
\]
are elliptic curves. We denote by $\overline{E}_t \subset \P^2$ the normalization of $E_t$.  Now fix any prime $p$ and consider the above family over the finite field $\F_p$. Assume that $t \in \F_p^{\times}$ is such that $\overline{E}_t$ is smooth. Let $z_t \in \Z_p$ be the respective Teichm\"{u}ller representative, that is the unique $p$-adic number satisfying $z_t^{p-1}=1$ and $z_t \con t \mod p$. One can show that 
\[
p+1-\# \overline{E}_t(\F_p) \con f_1(t) \mod p \,.
\]  
Hence $f_1(t)$ modulo $p$ is the Hasse invariant for this family, and we have $|f_1(z_t)|_p = 1$ precisely when the curve $\overline{E}_t$ is ordinary. The number $\omega(z_t)$ is then a reciprocal zero of the zeta function of $\overline{E}_t / \F_p$, i.e.
\[
{\mathcal Z}\Bigl( \overline{E}_t / \F_p; X \Bigr) \= \frac{(1-\omega(z_t)X)(1-\frac{p}{\omega(z_t)}X)}{(1-X)(1-pX)}\,.
\]
The reciprocal zero which is a $p$-adic unit is usually called the ``unit root'', which allows to distinguish between the two reciprocal roots in the case of ordinary reduction. We plan to devote another paper to the proof of such ``unit root formulas''. Lemma~\ref{dw_lemma} also shows that to get the first $s$ $p$-adic digits of the unit root it is sufficient to compute $f_s(z_t)/f_{s-1}(z_t)$ modulo $p^s$.

This situation resembles the classical example with the Legendre family due to John Tate and Bernard Dwork (see $\S 8$ in \cite{Katz72}, $\S 5$ in \cite{Dw62}). Theorem~\ref{Th} along with Lemma~\ref{dw_lemma} constitute a step towards proving such ``unit root formulas'' for families of hypersurfaces.

\section{Ghost terms}\label{sec:ghost}

For a Laurent polynomial $\L(x) \in \Z_p[x_1^{\pm1},\dots,x_d^{\pm1}]$ and an integer $m \ge 1$ we write $\L(x^m)$ for $\L(x_1^m,\dots,x_d^m)$. The Newton polyhedron of $\L$ is the convex hull in $\R^d$ of the exponent vectors of the monomials of $\L$. It is denoted by $\Newt(\L)$.

\begin{definition} For a Laurent polynomial $\L(x)$ and an integer $s \ge 1$ the ghost term $R_s(\L)$ is the Laurent polynomial defined by
\[
R_s(\L) (x) \;:=\; \L(x)^{p^s} - \L(x^p)^{p^{s-1}}\,.
\]
In addition we put $R_0(\L):=\L$.
\end{definition}

\begin{proposition} For every integer $s \geq 0$ one has
\begin{itemize}
\item[(i)] $\Lambda(x)^{p^s} \= R_0(\Lambda)(x^{p^s}) + R_1(\Lambda)(x^{p^{s-1}}) + \cdots + R_s(\Lambda)(x)$;
\item[(ii)] $ R_s(\L) \con 0 \mod \; p^s$;
\item[(iii)] $\Newt(R_s(\L)) \subset p^s \, \Newt(\L)$.
\end{itemize}
\end{proposition}
\begin{proof} Formula (i) easily follows by induction. (ii) is trivial when $s=0$ and clearly $p | \L(x)^p - \L(x^p)$ which proves the statement for $s=1$. To do induction in $s$ we use the fact that in any ring $R$ if $X\equiv Y \pmod{p^s}$ then $X^p\equiv Y^p \pmod{p^{s+1}}$. This proves (ii). (iii) follows from the definition of ghost terms and the following two obvious properties of the Newton polyhedron: for any Laurent polynomial $\Phi$ and any integer $m \ge 1$ one has $\Newt(\Phi(x)^m) \subset m \; \Newt(\Phi(x))$ and $\Phi(x^m) \subset m \; \Newt(\Phi(x))$.  
\end{proof}

Expanding any positive integer $n$ to the base $p$ as $n = n_0 + n_1 p + \dots + n_{\ell(n)-1} p^{\ell(n)-1}$ with digits $0 \leq n_i \leq p-1$ we use (i) in the above proposition to decompose the product
\[
\Lambda^{n} = \Lambda^{n_0} (\Lambda^{n_1})^p \ldots (\Lambda^{n_{\ell(n)-1}})^{p^{\ell(n)-1}}
\]
as the sum of products of ghost terms of the collection of $p$ Laurent polynomials $\Lambda^{a}$, $0 \le a \le p-1$. We obtain that $\Lambda(x)^{n}$ is the sum of the products
\[
R_{m,\L}^n(x) := \prod_{i=0}^{\ell(n)-1} R_{m_i}(\L^{n_i})(x^{p^{i-m_i}}) 
\]
where $m=(m_0,m_1,\ldots,m_{\ell(n)-1})$ runs over the set of all integral tuples of length $\ell(n)$ satisfying $0 \leq m_i \le i$. For such a tuple we denote $|m|=\sum m_i$. Then one has
\[
R_{m,\Lambda}^n(x) \con 0 \mod\; p^{|m|}  
\]
from (ii) in the above proposition, and (iii) gives us
\[
\Newt\Bigl(R_{m,\Lambda}^n \Bigr) \subset n \; \Newt(\Lambda)
\]
respectively.

\section{Indecomposable tuples}\label{sec:indec}
Denote the set of all tuples $(m_0,m_1,\ldots,m_{k-1}) \in \Z^k$ satisfying $0\leq m_i \leq i$ by $S_{k}$. Put $S=\cup_{k>0} S_k$. For $m' \in S_k$, $m'' \in S_l$ we denote $m' * m'' = (m'_0,\ldots,m'_{k-1},m''_0,\ldots,m''_{l-1}) \in S_{k+l}$.

\begin{definition}
A tuple $m\in S$ is called indecomposable if it cannot be presented as $m' * m''$ for $m', m''\in S$. The set of all indecomposable tuples of length $k$ is denoted as $S_k^{ind}$ and we put $S^{ind}=\cup_{k>0} S_k^{ind}$.
\end{definition}

Recall the notation $|m|=\sum m_i$ for a tuple $m \in S$. We have
\begin{proposition}\label{indec_lemma}
If $m\in S_k^{ind}$, then $|m|\geq k-1$.
\end{proposition}
\begin{proof}
If $m$ is indecomposable then for each $i\in\{1,\ldots,k-1\}$ there exists $j\geq i$ such that $m_j>j-i$, i.e. $j\geq i>j-m_j$. The number of such $i$ for a given $j$ is $m_j$. The total number of $i$ is $k-1$, therefore the sum of $m_j$ is at least $k-1$.
\end{proof}

For a Laurent polynomial $\Lambda$, integer $n  \ge 1$ and tuple $m \in S_{\ell(n)}$ we defined in the previous section the product of ghost terms $R_{m,\Lambda}^n$, so that $\L^n = \sum_{m \in S_{\ell(n)}} R_{m,\Lambda}^n$. Now we introduce Laurent polynomials
\[
I_\Lambda^n := \sum_{m\in S_{\ell(n)}^{ind}} R_{m, \Lambda}^n\,.
\]
We summarize their properties in the following

\begin{proposition}\label{indec_prop} For every integer $n  \ge 1$ one has
\begin{itemize} 
\item[(i)] $\L(x)^n = \sum \limits_{n=n^{(1)}*\cdots *n^{(r)}} I_\Lambda^{n^{(1)}}(x) \; I_\Lambda^{n^{(2)}}(x^{p^{\ell(n^{(1)})}}) \; \cdots \; I_\Lambda^{n^{(r)}}(x^{p^{\ell(n^{(1)})+\ell(n^{(2)})+\cdots+\ell(n^{(r-1)})}})$ where the sum runs over all $1 \le r \le \ell(n)$ and all possible partitions of the expansion of $n$ to the base $p$ into $r$ expansions of non-negative integers;
\item[(ii)] $I_\Lambda^n \con 0 \mod \; p^{\ell(n)-1}$;
\item[(iii)] $\Newt(I_\Lambda^n) \subset n \, \Newt(\L)$.
\end{itemize}
\end{proposition}
\begin{proof} 
We start with the formula $\L^n = \sum_{m \in S_{\ell(n)}} R_{m,\Lambda}^n$ of the previous section. Each tuple $m$ can be uniquely represented as a concatenation of indecomposable ones, $m=m^{(1)}*\cdots*m^{(r)}$ and we write $n$ in the form 
\[
n = n^{(1)} + p^{\ell(m^{(1)})} n^{(2)} + \ldots +p^{\ell(m^{(1)})+\ell(m^{(2)})+\cdots+\ell(m^{(r-1)})} n^{(r)}
\]
with $\ell(n^{(i)})\leq\ell(m^{(i)})$. Whenever $\ell(n^{(i)})<\ell(m^{(i)})$ the corresponding summand $R_{m,\L}^n(x)$ vanishes because in this case $\ell(m^{(i)})\geq 2$, so the last element of $m^{(i)}$ is not zero and the product in the definition of $R_{m,\L}^n(x)$ contains $R_{s}(\Lambda^0)=R_{s}(1)=0$ for $s>0$. Therefore we can assume $\ell(n^{(i)})=\ell(m^{(i)})$. In this case $n=n^{(1)}*\cdots*n^{(r)}$, the corresponding summand is written as 
\[
R_{m,\L}^n(x) \= R_{m^{(1)},\L}^{n^{(1)}}(x) R_{m^{(2)},\L}^{n^{(2)}}(x^{p^{\ell(n^{(1)})}}) \cdots R_{m^{(r)},\L}^{n^{(r)}}(x^{p^{\ell(n^{(1)})+\ell(n^{(2)})+\cdots+\ell(n^{(r-1)})}}),
\]
and (i) follows. (ii) follows from Proposition~\ref{indec_lemma} since $R_{m,\Lambda}^n(x) \con 0 \mod\; p^{|m|}$, and (iii) is due to the fact that $\Newt\Bigl(R_{m,\Lambda}^n \Bigr) \subset n \; \Newt(\Lambda)$.
\end{proof}

\section{The case of one interior point}\label{sec:onept}

In this section we proceed to compute free terms of powers of $\Lambda(x)=\Lambda(x_1,\ldots,x_d)$. We will work under the assumption that the origin is the only interior integral point of the Newton polyhedron of $\L$.

\begin{proposition} If $0 = (0,\ldots,0)$ is the only interior integral point of $\Newt(\L)$, then for any $r \ge 1$ and non-negative integers $n^{(1)},\dots,n^{(r)}$ one has 
\[
\Bigl[ I_\Lambda^{n^{(1)}}(x) \; I_\Lambda^{n^{(2)}}(x^{p^{\ell(n^{(1)})}}) \; \cdots \; I_\Lambda^{n^{(r)}}(x^{p^{\ell(n^{(1)})+\ell(n^{(2)})+\cdots+\ell(n^{(r-1)})}}) \Bigr]_0 = \prod_{i=1}^r \Bigl[ I_\Lambda^{n^{(i)}} \Bigr]_0.
\]
\end{proposition}
\begin{proof}
Since $\Newt(I_\Lambda^{n^{(1)}}) \subset n^{(1)} \; \Newt(\L)$ and $n^{(1)} < p^{\ell(n^{(1)})}$ we see that $N(I_\Lambda^{n^{(1)}})$ does not contain points of the lattice $p^{\ell(n^{(1)})} \Z^d$ other then $0$. Therefore the only contribution to the constant term on the left comes from the product
\begin{multline*}
\Bigl[ I_\Lambda^{n^{(1)}}(x) \; I_\Lambda^{n^{(2)}}(x^{p^{\ell(n^{(1)})}}) \; \cdots \; I_\Lambda^{n^{(r)}}(x^{p^{\ell(n^{(1)})+\ell(n^{(2)})+\cdots+\ell(n^{(r-1)})}})  \Bigr]_0\\
 = \Bigl[  I_\Lambda^{n^{(1)}}(x) \Bigr]_0 \Bigl[ I_\Lambda^{n^{(2)}}(x^{p^{\ell(n^{(1)})}}) \; \cdots \; I_\Lambda^{n^{(r)}}(x^{p^{\ell(n^{(1)})+\ell(n^{(2)})+\cdots+\ell(n^{(r-1)})}}) \Bigr]_0\,.
\end{multline*}
Thus by induction on $r$ we prove the statement.
\end{proof}

Together with Proposition~\ref{indec_prop} (i) this implies

\begin{corollary} $\Bigl[ \L^n \Bigr]_0 \= \sum \limits_{n=n^{(1)}*\cdots *n^{(r)}} \prod \limits_{i=1}^r \Bigl[ I_\Lambda^{n^{(i)}} \Bigr]_0$.
\end{corollary}

Now we are in a position to prove Lemma~\ref{lem}.

\begin{proof}[Proof of Lemma~\ref{lem}] Put $c_n = \Bigl[I_{\L}^n \Bigr]_0$. Then~\eqref{A1} is precisely the statement of the latter corollary, and~\eqref{A2} is given by Proposition~\ref{indec_prop} (ii).\end{proof}

\section{Proof of Theorem~\ref{Th}}\label{sec:proof}
\begin{proof}
We will prove~\eqref{c2}. Fixing $N$ and collecting coefficients near $X^N$ on both sides we see that what we need to prove is
\[
\underset{\bal & n + p m = N \\ \ell(n) \le&s+1,\ell(m) \le s-1 \eal} \sum b_n b_m \quad \equiv \underset{\bal & n' + p m' = N \\ &\ell(n'),\ell(m') \le s \eal} \sum b_{n'} b_{m'} \mod p^s
\] 
where the sums run over all pairs $(n,m)$ and $(n',m')$ that satisfy the respective conditions on the left and on the right. The sum of terms on the left with $\ell(n) \le s$ is equal to the sum of terms on the right with $\ell(m')\le s-1$ as the map $(n,m) \mapsto (n',m')=(n,m)$ provides a bijective correspondence. Therefore it remains to show that
\be{one_coeff}
\underset{\bal & n + p m = N \\ \ell(n&)=s+1,\ell(m) \le s-1 \eal} \sum b_n b_m \quad \equiv \underset{\bal & n' + m' p = N \\ \ell(&n') \leq s,\ell(m') = s \eal} \sum b_{n'} b_{m'} \mod p^s\,.
\ee 

Using decomposition~\eqref{A1} a product $b_n b_m$ becomes 
\be{bnbm}
b_n \, b_m \= \sum_{\bal n = n^{(1)}*\dots*n^{(r)} \\ m = m^{(1)}*\dots*m^{(l)} \eal}  c_{n^{(1)}} \cdot \ldots \cdot c_{n^{(r)}} \cdot c_{m^{(1)}} \cdot \ldots \cdot c_{m^{(l)}} \,,
\ee
where we sum over all possible pairs of partitions of $n$ and $m$. 
Let us say that a pair of partitions is \emph{good} if for some $1 \le i < r$ one either has  
\[
\ell(n^{(1)})+\ldots+\ell(n^{(i)}) \= \ell(m^{(1)})+\ldots+\ell(m^{(j)})+1
\]
for some $0 \le j < l$ or 
\[
\ell(n^{(1)})+\ldots+\ell(n^{(i)}) \geq \ell(m)+1 \,.
\]
For a \emph{good} pair of partitions we take the smallest such $i$ and consider the pair $(n',m')$ constructed as follows. If the former of the two options takes place then we put $n'=n^{(1)}*\dots*n^{(i)}*m^{(j+1)}*\dots*m^{(l)}$, $m'=m^{(1)}*\dots*m^{(j)}*n^{(i+1)}*\dots*n^{(r)}$. In the latter case let $i'$ be the index of the last nonzero element of $n^{(1)},\ldots,n^{(i)}$ (they cannot be all zero as otherwise we would have chosen $i=1$). We put $n'=n^{(1)}*\dots*n^{(i')}$, $m'=m*0*\cdots*0*n^{(i+1)}*\dots*n^{(r)}$, where the number of zeroes to be inserted is $\ell(n^{(1)})+\ldots+\ell(n^{(i)}) - \ell(m)-1$. It is not hard to see that $n+pm=N$ implies $n' + p m' = N$, and clearly $\ell(m')=\ell(n)-1 = s$. For $\ell(n')$ we either have $\ell(n')=\ell(m)+1 \le s$ or 
\[
\ell(n') = \ell(n^{(1)})+\ldots+\ell(n^{(i')}) < \ell(n) = s+1\,.
\]
Therefore $(n',m')$ will occur in the right-hand sum in~\eqref{one_coeff}, and the same product of $c$'s will enter decomposition~\eqref{bnbm} for $b_{n'}b_{m'}$. This way we obtain a bijective correspondence between \emph{good} pairs of partitions of $(n,m)$ in the left-hand sum in~\eqref{one_coeff} and \emph{good} pairs of partitions $n'=n'^{(1)}*\dots*n'^{(r')}$, $m'=m'^{(1)}*\dots*m'^{(l')}$ of $(n',m')$  in the right-hand sum, where the latter pair is called \emph{good} when for some $0 \le j < l'$ one either has  
\[
\ell(n'^{(1)})+\ldots+\ell(n'^{(i)}) \= \ell(m'^{(1)})+\ldots+\ell(m'^{(j)}) + 1
\]
for some $1 \le i < r'$ or 
\[
\ell(n') \leq \ell(m'^{(1)})+\ldots+\ell(m'^{(j)})+1 \,.
\] 

It remains to show that products of $c$'s corresponding to pairs of partitions on either side which are not \emph{good} vanish modulo $p^s$. Let us first consider left-hand pairs. Suppose a pair of partitions $n = n^{(1)}*\dots*n^{(r)}$, $m = m^{(1)}*\dots*m^{(l)}$ is not \emph{good}. There are $r-1$ possible sums $\ell(n^{(1)})+\ldots+\ell(n^{(i)})$ for $1 \le i < r$, $l$ possible sums $\ell(m^{(1)})+\ldots+\ell(m^{(j)})+1$ for $0 \le j < l$ and $s-\ell(m)$ numbers $k$ satisfying $\ell(m)+1 \leq k < \ell(n)=s+1$. As the pair of partitions is not \emph{good} all these numbers must be distinct. Since they all belong to the range between $1$ and $s$, we then have $(r-1)+l+(s-\ell(m)) \le s$, i.e. $r+l \le \ell(m)+1$. Using~\eqref{A2} we conclude that
\[
c_{n^{(1)}} \cdot \ldots \cdot c_{n^{(r)}} \cdot c_{m^{(1)}} \cdot \ldots \cdot c_{m^{(l)}} \con 0 \mod p^a 
\]
where 
\[\bal
a \= \sum_{i=1}^r \bigl( \ell(n^{(i)})-1 \bigr) + \sum_{j=1}^l \bigl( \ell(m^{(j)})-1 \bigr) &\= \ell(n)+\ell(m)-r-l\\
 &\ge \ell(n)-1 \= s.
\eal\]

Similarly, consider a pair of partitions $n'=n'^{(1)}*\dots*n'^{(r')}$, $m'=m'^{(1)}*\dots*m'^{(l')}$ which is not \emph{good}. There are now $r'-1$ possible sums $\ell(n'^{(1)})+\ldots+\ell(n'^{(i)})$ for $1 \le i < r'$, $l'$ possible sums $\ell(m'^{(1)})+\ldots+\ell(m'^{(j)})+1$ for $0 \le j < l'$ and $s+1-\ell(n')$ numbers $k$ satisfying $\ell(n') \leq k < \ell(m')+1=s+1$. All these numbers are distinct and belong to the range between $1$ and $s$, hence we have $r'-1+l'+(s+1-\ell(n')) \le s$, so $r'+l' \le \ell(n')$. Using~\eqref{A2} we conclude that
\[
c_{n'^{(1)}} \cdot \ldots \cdot c_{n'^{(r')}} \cdot c_{m'^{(1)}} \cdot \ldots \cdot c_{m'^{(l')}} \con 0 \mod p^a 
\]
where 
\[\bal
a \= \sum_{i=1}^{r'} \bigl( \ell(n'^{(i)})-1 \bigr) + \sum_{j=1}^{l'} \bigl( \ell(m'^{(j)})-1 \bigr) &\= \ell(n')+\ell(m')-r'-l'\\
 &\ge \ell(m') \= s.
\eal\]
\end{proof}

The reader could deduce congruences~\eqref{dig2} from Lemma~\ref{lem} in a similar way.

\section*{Acknowledgment}
We want to express our gratitude to Vasily Golyshev for introducing both authors to this topic and for discussions which motivated our work all the time.

\end{document}